\documentclass[11pt]{amsart}
\usepackage{
amssymb
,amsthm
,amsmath
,mathtools
,mathdots
}
\usepackage[all]{xy}
\usepackage[top=35truemm,bottom=35truemm,left=30truemm,right=30truemm]{geometry}
\usepackage[usenames]{color}

\SelectTips{cm}{12}
\objectmargin={1pt}

\newcommand{\C}{\mathbb{C}}

\newcommand{\PP}{\mathcal{P}}
\newcommand{\Sc}{\mathcal{S}}
\newcommand{\NN}{\mathcal{N}}

\newcommand{\GL}{\mathrm{GL}}
\newcommand{\SL}{\mathrm{SL}}
\newcommand{\SO}{\mathrm{SO}}
\newcommand{\Sp}{\mathrm{Sp}}
\renewcommand{\O}{\mathrm{O}}

\newcommand{\oo}{\mathfrak{o}}
\newcommand{\g}{\mathfrak{g}}
\newcommand{\h}{\mathfrak{h}}
\renewcommand{\sl}{\mathfrak{sl}}

\newcommand{\Irr}{\mathrm{Irr}}
\newcommand{\temp}{\mathrm{temp}}
\newcommand{\tr}{\mathrm{tr}}
\newcommand{\st}{\mathrm{st}}
\newcommand{\Ad}{\mathrm{Ad}}
\newcommand{\orth}{\mathrm{orth}}
\newcommand{\symp}{\mathrm{symp}}

\newcommand{\iif}{&\quad&\text{if }}
\newcommand{\other}{&\quad&\text{otherwise}}
\newcommand{\resp}{resp.~}
\renewcommand{\1}{\mathbf{1}}

\newcommand{\tl}[1]{\widetilde{#1}}
\newcommand{\pair}[1]{\langle #1 \rangle}

\theoremstyle{plain}
\newtheorem{thm}{Theorem}[section]
\newtheorem{lem}[thm]{Lemma}
\newtheorem{prop}[thm]{Proposition}
\newtheorem{cor}[thm]{Corollary}
\newtheorem{conj}[thm]{Conjecture}
\newtheorem{hyp}[thm]{Hypothesis}

\theoremstyle{definition}
\newtheorem{rem}[thm]{Remark}

\title{Endoscopic transfer and the wavefront upper bound conjecture}
\author{Hiraku Atobe and Dan Ciubotaru}
\date{}
\address{
Department of Mathematics, Kyoto University, 
Kitashirakawa-Oiwake-cho, Sakyo-ku, Kyoto, 606-8502, Japan
}
\email{
atobe@math.kyoto-u.ac.jp
}
\address{
Mathematical Institute, University of Oxford, Oxford OX2 6GG, UK
}
\email{
dan.ciubotaru@maths.ox.ac.uk
}
\allowdisplaybreaks

\begin{document}
\begin{abstract}
We verify the local analogue of Jiang's conjecture for the upper bound of 
the geometric wavefront sets of 
Arthur type representations of split classical $p$-adic groups with $p\gg 0$, 
under a certain condition. 
As a consequence, 
we also obtain the upper bound conjecture of Kim and the second author, 
and Hazeltine--Liu--Lo--Shahidi, 
under the same assumptions.
The proof uses Waldspurger's work on the endoscopic transfer supplemented by results of Konno and Varma, 
as well as the wavefront set computations in the unipotent case by Mason-Brown--Okada and the second author.
\end{abstract}

\maketitle

\section{Introduction}
This work is motivated by the elusive relation 
between the irreducible admissible Harish-Chandra characters of a reductive $p$-adic group 
and the Langlands-Arthur parameters. 
Arguably the best-known instance of this relation is the Hiraga--Ichino--Ikeda formal degree formula 
for square-integrable characters in terms of adjoint $\gamma$-factors \cite{HII}. 
The formal degree concerns the coefficient of the distribution attached to the zero nilpotent orbit 
in the Harish-Chandra--Howe local character expansion. 
At the opposite end, 
the largest orbits that contribute to the character expansion determine the wavefront set of the character distribution. 
They are the focus of the present paper.
 Assuming $p$ large (see Theorem \ref{main} for the precise conditions on $p$), we prove that  
the orbits of maximal dimension in the union of the geometric wavefront sets of the representations 
in {\it any} $A$-packet for a split classical $p$-adic group 
{\it equal} the Spaltenstein dual of the nilpotent orbit given by the $A$-parameter. 
\par

Certain relations of this sort, Conjectures \ref{upper} and \ref{jiang} , 
were proposed in \cite{CK,HLLS} and \cite{LS}, respectively. 
The former was motivated by the results of \cite{CMBO1,CMBO2} for unipotent representations of split (and `inner-to-split') $p$-adic groups, whereas the latter is a local analogue of Jiang's conjecture \cite{J}. 
The fact that such a link should exist had already been apparent from 
the foundational work of Adams--Barbasch--Vogan on the definition and construction of microlocal Arthur packets \cite{ABV}. 
In Theorem \ref{main}, we also show that the full Conjectures \ref{upper} and \ref{jiang} 
hold by further assuming Hypothesis \ref{hypo}.
\vskip 10pt

We fix a non-archimedean local field $F$ of characteristic zero and of residue characteristic $p \gg 0$.
Let $\overline{F}$ be a fixed algebraic closure of $F$. 
In this section, we state our main theorem (Theorem \ref{main})
and explain our idea for the proof. 

\subsection{Wavefront set}
Let $H$ be a connected reductive group over $F$ with the Lie algebra $\h$.
We denote by $\NN(\h(F))$ (\resp $\NN(\h(\overline{F}))$) 
the set of nilpotent $\Ad(H(F))$-orbits (\resp $\Ad(H(\overline{F}))$-orbits) 
in $\h(F)$ (\resp $\h(\overline{F})$).
For an irreducible admissible representation $\pi$ of $H(F)$, 
we denote its distribution character by $\Theta_\pi$. 
The Harish-Chandra--Howe local character expansion (\cite[Theorem 16.2]{HC}) states that 
\[
\Theta_\pi(\exp(X)) = \sum_{\oo_H \in \NN(\h(F))} c_{\oo_H}(\pi) \hat\mu_{\oo_H}(\exp(X))
\]
for $X \in \h(F)$ sufficiently near to $0$, 
where $\hat\mu_{\oo_H}$ is the Fourier transform of the orbital integral, 
and $c_{\oo_H}(\pi)$ is a constant. 
Set $\NN(\pi) = \{ \oo_H \in \NN(\h(F)) \;|\; c_{\oo_H}(\pi)\not= 0\}$ 
and let $\bar\NN(\pi)$ be the set of $\oo_H^\st \in \NN(\h(\overline{F}))$
such that $\oo_H^\st$ contains a rational orbit $\oo_H$ belonging to $\NN(\pi)$.
We denote by $\bar\NN(\pi)^{\max}$ the subset of $\bar\NN(\pi)$ consisting of maximal orbits
with respect to the closure ordering
\[
\oo_H^\st \leq \oo_H'^\st \overset{\text{def}}{\iff} \oo_H^\st \subset \overline{\oo_H'^\st}.
\]
We call $\bar\NN(\pi)^{\max}$ the \emph{(geometric) wavefront set} of $\pi$. 

\subsection{$L$-parameters and $A$-parameters}
Denote by $H^\vee$ the complex Langlands dual group, 
and by $\h^\vee$ the Lie algebra of $H^\vee$. 
Let $W_F$ be the Weil group of $F$. 
An \emph{$A$-parameter} for $H(F)$ is a homomorphism 
\[
\psi \colon W_F \times \SL_2(\C) \times \SL_2(\C) \rightarrow H^\vee
\]
such that 
$\psi(W_F)$ consists of semisimple elements in a bounded subset of $H^\vee$, 
$\psi|_{W_F}$ is smooth and $\psi|_{\SL_2(\C) \times \SL_2(\C)}$ is algebraic. 
It is called \emph{tempered} if $\psi$ is trivial on the second $\SL_2(\C)$. 
Let $\Psi(H(F))$ be the set of $H^\vee$-conjugacy classes of $A$-parameters for $H(F)$, 
and let $\Psi_\temp(H(F))$ be its subset consisting of tempered $A$-parameters. 
\par

Similarly, an \emph{$L$-parameter} for $H(F)$ is a homomorphism 
\[
\phi \colon W_F \times \SL_2(\C) \rightarrow H^\vee
\]
such that $\phi(W_F)$ consists of semisimple elements, 
$\phi|_{W_F}$ is smooth and $\phi|_{\SL_2(\C)}$ is algebraic. 
It is called \emph{tempered} if $\phi(W_F)$ is bounded. 
Under identifying $W_F \times \SL_2(\C)$ with $W_F \times \SL_2(\C) \times \{\1\}$, 
the tempered $L$-parameters are the same as the tempered $A$-parameters. 
Let $\Phi(H(F))$ be the set of $H^\vee$-conjugacy classes of $L$-parameters for $H(F)$, 
and let $\Phi_\temp(H(F))$ be its subset consisting of tempered $L$-parameters. 
Hence $\Psi_\temp(H(F)) = \Phi_\temp(H(F))$. 
\par

For $\psi \in \Psi(H(F))$, by the derivative of $\psi|_{\SL_2(\C) \times \SL_2(\C)}$, 
we have a linear map
\[
d\psi|_{\SL_2(\C) \times \SL_2(\C)} \colon \sl_2(\C) \times \sl_2(\C) \rightarrow \h^\vee.
\]
We set 
\[
N_{\psi} = 
d\psi\left(\1, \begin{pmatrix}
0 & 1 \\ 0 & 0
\end{pmatrix}\right) 
\in \h^\vee. 
\]
The $\Ad(H^\vee)$-orbit of $N_\psi$ is denoted by $\Ad(H^\vee) N_\psi$.
\par 

Let $\NN(\h^\vee)$ be the set of nilpotent $\Ad(H^\vee)$-orbits in $\h^\vee$.
We denote by 
\[
d_{H} \colon \NN(\h(\overline{F})) \rightarrow \NN(\h^\vee), \quad
d_{H^\vee} \colon \NN(\h^\vee) \rightarrow \NN(\h(\overline{F}))
\]
the Spaltenstein duality maps (\cite[Section 10]{Sp}). 
The images of these maps are the sets of \emph{special nilpotent conjugacy classes}.
Moreover, $\oo_H^\st \leq d_{H^\vee}(d_H(\oo_H^\st))$, 
and the equality holds if and only if $\oo_H^\st$ is special. 
See \cite[Appendix A]{BV} for more details.
An $A$-parameter $\psi \in \Psi(H(F))$ gives 
a nilpotent orbit 
\[
d_{H^\vee}(\Ad(H^\vee) N_\psi) \in \NN(\h(\overline{F})). 
\]

\subsection{$L$-packets and $A$-packets}
Now suppose that $H$ is a split classical group over $F$. 
Namely, $H$ is a symplectic group or a special orthogonal group.
Let $\Irr(H(F))$ be the set of equivalence classes of irreducible representations of $H(F)$, 
and let $\Irr_\temp(H(F))$ be its subset consisting of tempered representations. 
\par

Arthur \cite[Theorem 2.2.1]{Ar} defines 
the notion of the \emph{$A$-packet} $\Pi_\psi$ associated to $\psi \in \Psi(H(F))$.
This is a finite (multi-)set over $\Irr(H(F))$. 
Moreover, if $\phi \in \Psi_\temp(H(F)) = \Phi_\temp(H(F))$, 
the $A$-packet $\Pi_\phi$ is a subset of $\Irr_\temp(H(F))$, 
and 
\[
\Irr_\temp(H(F)) = \bigsqcup_{\phi \in \Phi_\temp(H(F))} \Pi_\phi.
\]
Combining the Langlands classification, 
one can define the \emph{$L$-packet} associated to $\phi \in \Phi(H(F))$, such that 
\[
\Irr(H(F)) = \bigsqcup_{\phi \in \Phi(H(F))} \Pi_\phi.
\]
We refer to this decomposition as the \emph{local Langlands correspondence} for $H(F)$.
\par

For $\psi \in \Psi(H(F))$, 
its \emph{dual $A$-parameter} $\widehat\psi$
is defined by 
\[
\widehat\psi(w,g_1,g_2) = \psi(w,g_2,g_1).
\]
We denote the Zelevinsky--Aubert dual of $\pi \in \Irr(H(F))$ by $\hat\pi$. 
As explained in \cite[Lemma 4.4.4 (1)]{AGIKMS}, we have
\[
\Pi_{\widehat\psi} = \{\hat\pi \;|\; \pi \in \Pi_\psi\}.
\]
\par

For $\phi \in \Phi(H(F))$, 
one can also define $\widehat\phi \colon W_F \times \SL_2(\C) \times \SL_2(\C) \rightarrow H^\vee$ by $\widehat\phi(w,g_1,g_2) = \phi(w,g_2)$ 
although $\widehat\phi$ might not be an $A$-parameter. 
We define $N_{\widehat\phi} \in \h^\vee$ by the same formula in the previous subsection. 
Namely, it is the image of 
$\begin{pmatrix}
0 & 1 \\ 0 & 0
\end{pmatrix}$
under the derivative $d\phi|_{\SL_2(\C)} \colon \sl_2(\C) \rightarrow \h^\vee$.

\subsection{Conjectures and the main result}
The following conjecture was independently formulated by Kim and the second author \cite[Conjectures 1.1, 1.9]{CK}
and Hazeltine--Liu--Lo--Shahidi \cite[Conjecture 1.1]{HLLS}. 

\begin{conj}[The upper bound conjecture]\label{upper}
Let $\phi \in \Phi(H(F))$. 
For any $\pi \in \Pi_\phi$, 
if $\oo_H^\st \in \bar\NN(\hat\pi)$, then 
we have
\[
\oo_H^\st \leq d_{H^\vee}(\Ad(H^\vee)N_{\widehat\phi}).
\]
Moreover, if $\phi$ is tempered, 
then $\bar\NN(\hat\pi)^{\max} = \{d_{H^\vee}(\Ad(H^\vee)N_{\widehat\phi})\}$ for some $\pi \in \Pi_\phi$.
\end{conj}

This is relevant for any connected reductive group $H$ 
provided the local Langlands correspondence for $H(F)$ is assumed. 
This conjecture is known in the following cases. 

\begin{itemize}
\item
$H=\GL_m$, \cite{MW}, see \cite[Theorem 1.4]{CK}. 

\item
$\pi$ is any depth-zero simple supercuspidal representation of a classical group (\cite[Theorem 1.8]{CK}). 

\item
{$H$ is an inner form of a split group and} 
$\pi$ is any unipotent representation of $H(F)$ with real infinitesimal character
(\cite[Theorem 1.4.1]{CMBO2}).

\item
$H$ is an exceptional group $G_2$ (\cite[Theorem 1.10]{CK}).
\end{itemize}

\begin{rem}
An analogue of Conjecture \ref{upper} for covering groups 
has been formulated by Gao--Liu--Lo--Shahidi \cite[Conjecture 1.2]{GLLS}, and proved for the Kazhdan--Patterson cover of $\GL_m$ in \cite[Theorem 1.3]{GLLS}.
\end{rem}

Conjecture \ref{upper} is reduced to the case where $\phi$ is tempered 
by \cite[Proposition 2.2]{CK} or \cite[Theorem 1.2]{HLLS}. 
This case is a special case where $\psi = \widehat\phi$ of 
the local analogue of Jiang's conjecture \cite[Conjecture 4.2]{J} stated as follows 
(see \cite[Conjecture 1.7]{LS}).

\begin{conj}[Jiang's conjecture]\label{jiang}
Let $\psi \in \Psi(H(F))$. 
For any $\pi \in \Pi_\psi$, 
if $\oo_H^\st \in \bar\NN(\pi)$, then 
we have
\[
\oo_H^\st \leq d_{H^\vee}(\Ad(H^\vee)N_{\psi}).
\]
Moreover, there exists $\pi \in \Pi_\psi$ such that 
$\bar\NN(\pi)^{\max} = \{d_{H^\vee}(\Ad(H^\vee)N_{\psi})\}$.
\end{conj}

In this paper, we give a result toward Conjecture \ref{jiang}
and hence Conjecture \ref{upper} for split classical groups under some conditions. 
Our main result is stated as follows. 

\begin{thm}\label{main}
Let $H$ be a split classical group over $F$.
Suppose that one of the following holds: 
\begin{description}
\item[(B)]
$H = \SO_{2n+1}$ and $p > 6n+3$; 
\item[(C)]
$H = \Sp_{2n}$ and $p > 6n$; 
\item[(D)]
$H = \SO_{2n}$ and $p > 6n$. 
\end{description}
Let $\psi \in \Psi(H(F))$. 
Then we have:
\begin{enumerate}
\item
There exists $\pi \in \Pi_\psi$ such that 
\[
d_{H^\vee}(\Ad(H^\vee)N_\psi) \in \bar\NN(\pi).
\]
\item
For $\pi \in \Pi_\psi$ and for $\oo_H^\st \in \bar\NN(\pi)$, 
if $\oo_H^\st \not= d_{H^\vee}(\Ad(H^\vee)N_\psi)$, then 
\[
\dim(\oo_H^\st) < \dim(d_{H^\vee}(\Ad(H^\vee)N_\psi)).
\]
Here, $\dim(\oo_H^\st)$ is the Zariski dimension of the orbit as an algebraic variety.
\item
If we further assume Hypothesis \ref{hypo} below, 
then Conjecture \ref{jiang} is true.
In particular, Conjecture \ref{upper} also holds.
\end{enumerate}
\end{thm}

\subsection{Outline of the proof}\label{sec.outline}
Throughout the paper, we assume that $p \gg 0$ as in Theorem \ref{main}.
The proof of Theorem \ref{main} is similar to the proofs of the Generic Packet Conjecture 
by Konno \cite{K} and Varma \cite{V}. 
The main idea is to compare local character expansions via 
\emph{endoscopic character identities}.
\par

Fix $\psi \in \Psi(H(F))$.
We identify the $A$-packet $\Pi_\psi$ 
with the semisimple representation 
\[
\bigoplus_{\pi \in \Pi_\psi} \pi. 
\]
In particular, we consider its local character expansion 
\[
\Theta_{\Pi_\psi} = \sum_{\pi \in \Pi_\psi} \Theta_\pi
=  \sum_{\oo_H \in \NN(\h(F))} c_{\oo_H}(\Pi_\psi) \hat\mu_{\oo_H}
\]
with 
\[
c_{\oo_H}(\Pi_\psi) = \sum_{\pi \in \Pi_\psi}c_{\oo_H}(\pi).
\]
Define $\NN(\Pi_\psi) \subset \NN(\h(F))$ and $\bar\NN(\Pi_\psi) \subset \NN(\h(\overline{F}))$
similar to $\NN(\pi)$ and $\bar\NN(\pi)$, respectively. 
\par

Set $G = \GL_m$, where $m$ is the dimension of the standard representation $H^\vee \rightarrow G^\vee = \GL_m(\C)$.
By composing with this map, we can regard $\psi$ as an $A$-parameter for $G(F)$, 
and obtain an irreducible self-dual representation $\pi_\psi$ of $G(F)$. 
Define an involution $\theta$ on $G(F)$ by 
\[
\theta(g) = J{}^tg^{-1}J^{-1}, 
\quad
J = \begin{pmatrix}
&&&1 \\
&&-1&\\
&\iddots&& \\
(-1)^{m-1}&&
\end{pmatrix} \in G(F).
\]
We can extend $\pi_\psi$ to an irreducible representation $\tl\pi_\psi$ of $G \rtimes \pair{\theta}$.
The twisted character of $\tl\pi_\psi$ is denoted by $\Theta_{\tl\pi_\psi}$. 
\par

Let $\Sc_\psi$ be the component group for $\psi$, 
and let $\widehat\Sc_\psi$ be its Pontryagin dual. 
Recall that Arthur defined a map
\[
\Pi_\psi \rightarrow \widehat\Sc_\psi, \;
\pi \mapsto \pair{\cdot, \pi}_\psi. 
\]
Set $s_\psi = \psi(\1_{W_F}, \1, -\1)$ which is regarded as an element of $\Sc_\psi$.
In this setting, \cite[Theorem 2.2.1]{Ar} asserts that 
\begin{description}
\item[(T1)]
the distribution
\[
S\Theta_{\psi} = \sum_{\pi \in \Pi_\psi} \pair{s_\psi, \pi}_\psi \Theta_\pi
\]
is stable, and $\Theta_{\tl\pi_\psi}$ is the transfer of $S\Theta_\psi$; 
\item[(T2)]
Each $s \in \Sc_\psi$ with $s \not= 1$ gives an elliptic endoscopic group $H_1 \times H_2$ of $H$
and an $A$-parameter $\psi_i$ for $H_i$ 
(of the form $\psi_i = \widehat\phi_i$ for some $\phi_i \in \Phi_\temp(H_i)$) such that 
\[
S\Theta_{\psi, s} = \sum_{\pi \in \Pi_\psi} \pair{s \cdot s_\psi, \pi}_\psi \Theta_\pi
\]
is the transfer of $S\Theta_{\psi_1} \otimes S\Theta_{\psi_2}$. 
\end{description}
\par

The first step of the proof of Theorem \ref{main} is to establish an analogue for the twisted $\GL_m$ case.
By the twisted local character expansion established by Clozel \cite[Theorem 3]{C}, 
we can expand
\[
\Theta_{\tl\pi_\psi} 
= \sum_{\oo_G \in \NN(\g^\theta(F))} c_{\oo_G}(\tl\pi_\psi) \hat\mu_{\oo_G}. 
\]
See Section \ref{sec.TLCE} below for the notations. 
We define $\NN_\theta(\tl\pi_\psi) \subset \NN(\g^\theta(F))$ 
and $\bar\NN_\theta(\tl\pi_\psi) \subset \NN(\g^\theta(\overline{F}))$
similar to $\NN(\pi)$ and $\bar\NN(\pi)$, respectively. 
In Theorem \ref{twist} below, we will show that 
\[
\bar\NN_\theta(\tl\pi_\psi)^{\max} = \{d_{G^\vee}(\Ad(G^\vee) N_\psi)\}.
\]
This is an immediate consequence of 
results of Konno \cite{K}, Varma \cite{V}, and M{\oe}glin--Waldspurger \cite{MW}.
Moreover, Konno's result \cite[Theorem 4.1(2)]{K} says that  
the constant $c_{\oo_G}(\tl\pi_\psi)$ for $\oo_G \subset d_{G^\vee}(\Ad(G^\vee) N_\psi)$ 
does not depend on the choice of the rational orbit.
\par

Set $\oo_H^\st = d_{H^\vee}(\Ad(H^\vee) N_\psi)$ and $\oo_G^\st = d_{G^\vee}(\Ad(G^\vee) N_\psi)$.
The second step of the proof of Theorem \ref{main} is to show that 
there exists a constant $\gamma_{\oo_H}$ for each rational orbit $\oo_H \subset \oo_H^\st$ such that 
\[
\sum_{\oo_G \subset \oo_G^\st} \hat\mu_{\oo_G}
\quad
\text{is equal to the transfer of} 
\quad
\sum_{\oo_H \subset \oo_H^\st} \gamma_{\oo_H}\hat\mu_{\oo_H}
\]
on a small enough neighborhood of $\theta \in G(F) \rtimes \theta$.
See Theorem \ref{transfer} below. 
Hypothesis \ref{hypo} is an analogue of this result. 
These claims could be proven by similar arguments to the standard case explained in the next paragraph, 
but there is no explicit reference in the literature.
Thus we will prove Theorem \ref{transfer} in this paper. 
The proof uses Conjecture \ref{upper} in the case established in \cite{CMBO2},
which requires that $H$ is split over $F$.
\par

An analogous result for the standard endoscopy with unramified endoscopic groups
is known by Waldspurger \cite{W_pave}. 
Let $H_1 \times H_2$ be an unramified endoscopic group of $H$. 
For a special nilpotent orbit $\oo_{H_i}^\st \in \NN(\h_i(\overline{F}))$, 
Waldspurger explicitly defined a nilpotent orbit
\[
\oo_H^\st = W(\oo_{H_1}^\st, \oo_{H_2}^\st) \in \NN(\h(\overline{F})),
\]
and proved a certain transfer result in \cite[XII.9 Th\'eor\`eme]{W_pave}.
The third step of the proof of Theorem \ref{main} is to relate the Waldspurger map $W(\oo_{H_1}^\st, \oo_{H_2}^\st)$
and the Spaltenstein duality (Proposition \ref{WS}). 
\par

Using the two transfer results, we can relate the local character expansion of 
$\Theta_{\Pi_\psi} = S\Theta_{\psi,s_\psi}$
with the one of $\Theta_{\tl\pi_\psi}$ (\resp $S\Theta_{\psi_1} \otimes S\Theta_{\psi_2}$) 
if $s_\psi = 1$ (\resp $s_\psi \not= 1$).
Then by induction on $m$, 
we can deduce Theorem \ref{main} in Section \ref{sec.proof}. 
Notice that if $s = s_\psi \not= 1$, 
then the corresponding endoscopic group $H_1 \times H_2$ is split over $F$
so that we can apply \cite[XII.9 Th\'eor\`eme]{W_pave}. 

\begin{rem}\label{OvsSO}
The local Langlands correspondence for $H = \SO_{2n}$ as proven by Arthur, 
is given up to $\O_{2n}(F)$-conjugation. 
Thus, in turn, we should replace $\Ad(H^\vee)N_\psi$ by $\Ad(\O_{2n}(\C))N_\psi$ in our results. 
We will ignore this discrepancy since it is not an important difference for the geometric wavefront set results.
By replacing $\SO_{2n}$ with $\O_{2n}$, we can use the following properties:
\begin{enumerate}
\item
the distributions on $\SO_{2n}(F)$ considered in the paper are all $\O_{2n}(F)$-invariant;
\item
the transfer of stable distributions on $H(F)$ to $G(F) \rtimes \theta$ is injective;
\item
$\Ad(G^\vee)\oo_{H^\vee} \cap H^\vee = \oo_{H^\vee}$ for $\oo_{H^\vee} \in \NN(\h^\vee)$.
\end{enumerate}
Note that (2) and (3) hold when $H = \SO_{2n+1}$ or $H = \Sp_{2n}$. 
\end{rem}

\section{Twisted endoscopy}
Let $H$ be a split classical group over $F$, and $G = \GL_m$, 
where $m$ is the dimension of the standard representation of $H^\vee$. 
In this section, we compare the local character expansions of $S\Theta_\psi$ and $\Theta_{\tl\pi_\psi}$.

\subsection{Twisted local character expansion}\label{sec.TLCE}
Let $\theta$ be the involution on $G$ defined in Section \ref{sec.outline}. 
In this subsection, we will establish the twisted analogue of \cite[Theorem 1.4]{CK} 
for irreducible self-dual representations of $G(F)$ which are co-tempered. 
\par

Let $(\pi, V)$ be an irreducible self-dual representation of $G(F) = \GL_m(F)$. 
Then we can choose an operator $\tl\pi(\theta)$ on $V$ with $\tl\pi(\theta)^2 = \1$ 
such that 
\[
\tl\pi(\theta) \circ \pi(g) = \pi(\theta(g)) \circ \tl\pi(\theta), \quad \forall g \in G(F).
\]
The \emph{twisted distribution character} $\Theta_{\tl\pi}$ of $\pi$ is 
a distribution on $C_c^\infty(G(F) \rtimes \theta)$ given by 
\[
\Theta_{\tl\pi}(f) = \tr\left(
\int_{G(F)} f(g \theta) \pi(g) \circ \tl\pi(\theta) dg
\right).
\]
\par

Let $\g$ be the Lie algebra of $G$.
The group (\resp Lie algebra) of $\theta$-fixed points of $G$ (\resp $\g$) is denoted by $G^\theta$ (\resp $\g^\theta$).
We denote by $\NN(\g^\theta(F))$ 
the set of nilpotent $\Ad(G^\theta(F))$-orbits in $\g^\theta(F)$.
Clozel showed that $\Theta_{\tl\pi}$ has a twisted local character expansion 
\[
\Theta_{\tl\pi} = \sum_{\oo_G \in \NN(\g^\theta(F))} c_{\oo_G}(\tl\pi) \hat\mu_{\oo_G}
\]
on a small enough neighborhood of $\theta \in G(F) \rtimes \theta$
for some unique complex numbers $c_{\oo_G}(\tl\pi) \in \C$, 
where $\hat\mu_{\oo_G}$ is the Fourier transform of the twisted orbital integral associated to $\oo_G$. 
See \cite[Theorem 3]{C} for more details. 
Set $\NN_\theta(\tl\pi) = \{ \oo_G \in \NN(\g^\theta(F)) \;|\; c_{\oo_G}(\tl\pi) \not= 0\}$. 
Note that this set does not depend on the choice of $\tl\pi(\theta)$.
\par

Let $\NN(\g^\theta(\overline{F}))$ be 
the set of nilpotent $\Ad(G^\theta(\overline{F}))$-orbits in $\g^\theta(\overline{F})$. 
Note that there is a canonical injection 
\[
\NN(\g^\theta(\overline{F})) \hookrightarrow \NN(\g(\overline{F})) \cong \NN(\g(F)). 
\]
We regard $\NN(\g^\theta(\overline{F}))$ as a subset of $\NN(\g(\overline{F}))$ via this inclusion. 
Let $\bar\NN_\theta(\tl\pi)$ be the set of $\oo_G^\st \in \NN(\g^\theta(\overline{F}))$ 
such that $\oo_G^\st$ contains a rational orbit $\oo_G$ belonging to $\NN_\theta(\tl\pi)$. 
We denote by $\bar\NN_\theta(\tl\pi)^{\max}$ the subset of $\bar\NN_\theta(\tl\pi)$ 
consisting of maximal orbits with respect to the closure ordering in $\NN(\g(\overline{F}))$. 
We call $\bar\NN_\theta(\tl\pi)^{\max}$ the \emph{(geometric) wavefront set} of $\tl\pi$.
\par

Let $\psi \colon W_F \times \SL_2(\C) \times \SL_2(\C) \rightarrow G^\vee$ 
be an $A$-parameter for $G(F) = \GL_m(F)$.
It gives an irreducible representation $\pi_\psi$ of $G(F)$, 
which is a product of unitary Speh representations.

\begin{thm}\label{twist}
Let $\psi \colon W_F \times \SL_2(\C) \times \SL_2(\C) \rightarrow G^\vee$ 
be a self-dual tempered $A$-parameter for $G(F) = \GL_m(F)$, 
and let $\pi_\psi \in \Irr(G(F))$ be the corresponding representation. 
Then 
\[
\bar\NN_\theta(\tl\pi_\psi)^{\max} = \{d_{G^\vee}(\Ad(G^\vee) N_\psi)\}.
\]
Moreover, $c_{\oo_G}(\tl\pi_\psi) = c_{\oo_G'}(\tl\pi_\psi)$ 
for any rational orbits $\oo_G$ and $\oo_G'$ contained in $d_{G^\vee}(\Ad(G^\vee) N_\psi)$.
\end{thm}
\begin{proof}
Recall that $\bar\NN(\pi_\psi)^{\max} = \{d_{G^\vee}(\Ad(G^\vee) N_\psi)\}$ 
by M{\oe}glin--Waldspurger \cite{MW} as explained in \cite[Theorem 1.4]{CK}.
Then Varma showed in \cite[Lemma 5.29]{V} that 
for any $\oo_G^\st \in \bar\NN_\theta(\tl\pi_\psi)^{\max}$, 
the corresponding degenerate Whittaker model is nonzero. 
Hence we have
\[
\oo_G^\st \leq d_{G^\vee}(\Ad(G^\vee) N_\psi).
\]
On the other hand, by Konno \cite[Theorem 4.1]{K}
together with the multiplicity one result of degenerate Whittaker models by M{\oe}glin--Waldspurger \cite{MW}
(cf., see \cite[Proposition 4.4]{K}),
we have
\[
d_{G^\vee}(\Ad(G^\vee) N_\psi) \in \bar\NN_\theta(\tl\pi_\psi)^{\max}. 
\]
Hence we obtain the first assertion.
The last assertion follows from \cite[Theorem 4.1(2), Remark 4.2]{K}.
\end{proof}

\subsection{Transfer of nilpotent orbital integrals}
To compare the local character expansions of $S\Theta_\psi$ and $\Theta_{\tl\pi_\psi}$, 
we establish a transfer result as follow. 

\begin{thm}\label{transfer}
Let $\oo_H^\st \in \NN(\h(\overline{F}))$ be a special nilpotent orbit.
Define $\oo_G^\st \in \NN(\g^\theta(\overline{F}))$ by 
\[
\oo_G^\st = d_{G^\vee}(\Ad(G^\vee)d_{H}(\oo_H^\st)). 
\]
Then for each rational orbit $\oo_H \subset \oo_H^\st$, 
there exists a constant $\gamma_{\oo_H} \in \C$ such that 
\[
\sum_{\oo_G \subset \oo_G^\st} \hat\mu_{\oo_G}
\quad\text{is equal to the transfer of }\quad
\sum_{\oo_H \subset \oo_H^\st} \gamma_{\oo_H} \hat\mu_{\oo_H}
\]
on a small enough neighborhood of $\theta \in G(F) \rtimes \theta$.
\end{thm}
\begin{proof}
Consider the unique $A$-parameter $\psi_0 \in \Psi_\temp(H(F))$
such that $\psi_0$ is trivial on $W_F \times \SL_2(\C) \times \{1\}$ and 
\[
\Ad(H^\vee) N_{\psi_0} = d_H(\oo_H^\st). 
\]
If we regard $\psi_0$ as an $L$-parameter for $G(F)$ 
via the standard representation $H^\vee \hookrightarrow G^\vee$, 
then the corresponding $\Ad(G^\vee)$-orbit in $\NN(\g^\vee)$ is 
$\Ad(G^\vee)N_{\psi_0} = \Ad(G^\vee)d_{H}(\oo_H^\st)$.
\par

We can apply \cite[Theorem 1.4.1]{CMBO2} to each $\pi \in \Pi_{\psi_0}$, 
and we obtain that
\[
\bar\NN(\pi)^{\max} = \{d_{H^\vee}(\Ad(H^\vee) N_{\psi_0})\} 
= \{d_{H^\vee} (d_H(\oo_H^\st))\} = \{\oo_H^\st\}. 
\]
Here, the last equation follows since $\oo_H^\st$ is special. 
In particular, by a result of M{\oe}glin--Waldspurger \cite{MW},
for each rational orbit $\oo_H \subset \oo_H^\st$, 
the coefficient $c_{\oo_H}(\pi)$ of $\hat\mu_{\oo_H}$ in the local character expansion of $\Theta_\pi$
is a non-negative real number, and it is nonzero for some $\oo_H \subset \oo_H^\st$. 
We set
\[
c_{\oo_H} = \sum_{\pi \in \Pi_{\psi_0}} c_{\oo_{H}}(\pi).
\]
Then there is a rational orbit $\oo_H \subset \oo_H^\st$ such that $c_{\oo_H} \not= 0$.
\par

Note that $s_{\psi_0} = 1$ in $\Sc_{\psi_0}$. 
Hence 
\[
\Theta_{\tl\pi_{\psi_0}} 
= \sum_{\oo_G \in \NN(\g^\theta(F))} c_{\oo_G}(\tl\pi_{\psi_0}) \hat\mu_{\oo_G}
\]
is the transfer of 
\[
S\Theta_{\psi_0} = \sum_{\pi \in \Pi_{\psi_0}} \Theta_\pi
= \sum_{\oo_H \in \NN(\h(F))} c_{\oo_H} \hat\mu_{\oo_H}.
\]
By the same argument as in the proof of \cite[Theorem 8.4]{K}
(cf., see also \cite[Remark 6.10]{V} and \cite[XII.8.8. Lemma]{W_pave}),
the homogeneous property implies a transfer result for each homogeneous components of $\Theta_{\tl\pi_{\psi_0}}$ and $S\Theta_{\psi_0}$
(after shifting the degrees).
The lowest degree term of $\Theta_{\tl\pi_{\psi_0}}$ corresponds to its wavefront set, 
which is
\[
d_{G^\vee}(\Ad(G^\vee)N_{\psi_0}) = d_{G^\vee}(\Ad(G^\vee)d_{H}(\oo_H^\st)) = \oo_G^\st
\]
by Theorem \ref{twist}. 
On the other hand, the lowest degree term of $S\Theta_{\psi_0}$ corresponds to $\oo_H^\st$. 
Therefore, by comparing the lowest degree parts, 
we see that 
\[
\sum_{\oo_G \subset \oo_G^\st} c_{\oo_G}(\tl\pi_{\psi_0}) \hat\mu_{\oo_G}
\quad\text{is equal to the transfer of}\quad
\sum_{\oo_H \subset \oo_H^\st} c_{\oo_H} \hat\mu_{\oo_H}
\]
on a small enough neighborhood of $\theta \in G(F) \rtimes \theta$.
Since $c_{\oo_G}(\tl\pi_{\psi_0})$ is a nonzero constant depending only on $\oo_G^\st$ (and $\tl\pi_{\psi_0}$), 
we obtain the assertion. 
\end{proof}

More generally, we expect the following assumption to hold; 
it should be possible that this could be proven by a similar argument as \cite[XII.9 Th\'eor\`eme]{W_pave}.

\begin{hyp}\label{hypo}
Let $\oo_H^\st \in \NN(\h(\overline{F}))$ be a special nilpotent orbit.
Define $\oo_G^\st \in \NN(\g^\theta(\overline{F}))$ by 
\[
\oo_G^\st = d_{G^\vee}(\Ad(G^\vee)d_{H}(\oo_H^\st)). 
\]
Then for any constants $\{c_{\oo_H} \;|\; \oo_H \subset \oo_H^\st\}$ such that 
$\sum_{\oo_H \subset \oo_H^\st} c_{\oo_H} \hat\mu_{\oo_H}$ is stable, 
there exist constants $\{a_{\oo_G} \;|\; \oo_G \subset \oo_G^\st\}$ such that 
\[
\sum_{\oo_G \subset \oo_G^\st} a_{\oo_G} \hat\mu_{\oo_G}
\quad\text{is equal to the transfer of}\quad
\sum_{\oo_H \subset \oo_H^\st} c_{\oo_H} \hat\mu_{\oo_H} 
\]
on a small enough neighborhood of $\theta \in G(F) \rtimes \theta$.
\end{hyp}

\subsection{Wavefront sets for $A$-packets}
Let $\psi \in \Psi(H(F))$, 
and we regard $\psi$ as an $A$-parameter for $G(F)$ 
via the standard representation $H^\vee \hookrightarrow G^\vee$. 
By \cite[Theorem 2.2.1]{Ar}, we know that
\[
\Theta_{\tl\pi_\psi} 
= \sum_{\oo_G \in \NN(\g^\theta(F))} c_{\oo_G}(\tl\pi_\psi) \hat\mu_{\oo_G}
\]
is the transfer of 
\[
S\Theta_\psi = \sum_{\pi \in \Pi_\psi} \pair{s_\psi,\pi}_\psi \Theta_\pi
= \sum_{\oo_H \in \NN(\h(F))} c_{\oo_H}(\psi) \hat\mu_{\oo_H}, 
\]
where we set
\[
c_{\oo_H}(\psi) = \sum_{\pi \in \Pi_\psi} \pair{s_\psi,\pi}_\psi c_{\oo_H}(\pi).
\]
We also write 
\[
S\Theta_\psi = \sum_{\oo_H^\st \in \NN(\h(\overline{F}))} 
\left( \sum_{\oo_H \subset \oo_H^\st} c_{\oo_H}(\psi) \hat\mu_{\oo_H}\right).
\]
Since $S\Theta_\psi$ is stable, by \cite[IX.15 Th\'eor\`eme]{W_pave}, 
only special nilpotent orbits $\oo_H^\st \in \NN(\h(\overline{F}))$ can contribute in this sum.
Set $\NN(\psi) = \{ \oo_H \in \NN(\h(F)) \;|\; c_{\oo_H}(\psi)\not= 0\}$ 
and let $\bar\NN(\psi)$ be the set of $\oo_H^\st \in \NN(\h(\overline{F}))$
such that $\oo_H^\st$ contains a rational orbit $\oo_H$ belonging to $\NN(\psi)$.
\par

As a consequence of Theorem \ref{transfer}, we obtain the following. 
\begin{cor}\label{cor1}
We have
\[
d_{H^\vee}(\Ad(H^\vee) N_\psi) \in \bar\NN(\psi).
\]
Moreover, for $\oo_H^\st \in \bar\NN(\psi)$, 
if $\oo_H^\st \not= d_{H^\vee}(\Ad(H^\vee) N_\psi)$, then 
\[
\dim(\oo_H^\st) < \dim(d_{H^\vee}(\Ad(H^\vee) N_\psi)).
\]
\end{cor}
\begin{proof}
Set $d = \dim(d_{H^\vee}(\Ad(H^\vee) N_\psi))$. 
As in the proof of Theorem \ref{transfer}, 
consider the unique $A$-parameter $\psi_0 \in \Psi(H(F))$
such that $\psi_0$ is trivial on $W_F \times \SL_2(\C) \times \{\1\}$ and 
\[
\Ad(H^\vee) N_{\psi_0} = \Ad(H^\vee) N_\psi. 
\]
For a constant $e \in \C^\times$
the transfer of $S\Theta_\psi - eS\Theta_{\psi_0}$ is 
\[
\Theta_{\tl\pi_\psi}-e\Theta_{\tl\pi_{\psi_0}}
= \sum_{\oo_G \in \NN(\g^\theta(F))} \left( c_{\oo_G}(\tl\pi_\psi)-e c_{\oo_G}(\tl\pi_{\psi_0}) \right) \hat\mu_{\oo_G}.
\]
Since $d_{G^\vee}(\Ad(G^\vee) N_\psi) = d_{G^\vee}(\Ad(G^\vee) N_{\psi_0})$,
by Theorem \ref{twist}, we can choose $e \in \C^\times$ such that 
for $\oo_G^\st \in \NN(\g^\theta(\overline{F}))$, 
if $c_{\oo_G}(\tl\pi_\psi)-e c_{\oo_G}(\tl\pi_{\psi_0}) \not= 0$ for some $\oo_G \subset \oo_G^\st$, 
then $\oo_G^\st < d_{G^\vee}(\Ad(G^\vee) N_\psi)$.
This together with Theorem \ref{transfer} implies that 
\[
\sum_{\substack{\oo_H^\st \in \NN(\h(\overline{F})) \\ \dim(\oo_H^\st) \geq d}} 
\left( \sum_{\oo_H \subset \oo_H^\st} c_{\oo_H}(\psi) \hat\mu_{\oo_H}\right)
=
e\sum_{\substack{\oo_H^\st \in \NN(\h(\overline{F})) \\ \dim(\oo_H^\st) \geq d}} 
\left( \sum_{\oo_H \subset \oo_H^\st} c_{\oo_H}(\psi_0) \hat\mu_{\oo_H}\right)
\]
By \cite[Theorem 1.4.1]{CMBO2}, 
in the right-hand side, only $\oo_H^\st = d_{H^\vee}(\Ad(H^\vee) N_\psi)$ can contribute. 
By the linear independence of $\hat\mu_{\oo_H}$ (\cite[Theorem 5.11]{HC}), 
the same holds for the left-hand side.
\end{proof}

Using Hypothesis \ref{hypo}, 
we can refine this corollary as follows. 

\begin{prop}\label{refine}
Assume Hypothesis \ref{hypo}.
Then for any $\oo_H^\st \in \bar\NN(\psi)$, 
we have 
\[
\oo_H^\st \leq d_{H^\vee}(\Ad(H^\vee) N_\psi).
\]
\end{prop}
\begin{proof}
For each special orbit $\oo_H^\st \in \NN(\h(\overline{F}))$, 
let $\{a_{\oo_G}(\psi) \;|\; \oo_G \subset d_{G^\vee}(\Ad(G^\vee)d_{H}(\oo_H^\st))\}$ 
be the constants in Hypothesis \ref{hypo}
given by $\{c_{\oo_H}(\psi) \;|\; \oo_H \subset \oo_H^\st\}$.
Then by considering the transfer of $S\Theta_\psi$, 
we see that the local character expansion of $\Theta_{\tl\pi_\psi}$ is given by
\[
\Theta_{\tl\pi_\psi} = \sum_{\oo_H^\st} 
\sum_{\oo_G \subset \oo_G^\st} a_{\oo_G}(\psi) \hat\mu_{\oo_G}, 
\]
where $\oo_H^\st$ runs over special nilpotent orbits in $\NN(\h(\overline{F}))$
and $\oo_G$ runs over rational orbits contained in
$\oo_G^\st = d_{G^\vee}(\Ad(G^\vee)d_{H}(\oo_H^\st))$.
In particular, if $\oo_H^\st \in \bar\NN(\psi)$, 
then $d_{G^\vee}(\Ad(G^\vee)d_{H}(\oo_H^\st)) \in \bar\NN_\theta(\tl\pi_\psi)$. 
By Theorem \ref{twist}, we have
\[
d_{G^\vee}(\Ad(G^\vee)d_{H}(\oo_H^\st)) \leq d_{G^\vee}(\Ad(G^\vee) N_\psi). 
\]
Hence $\Ad(G^\vee)d_{H}(\oo_H^\st) \geq \Ad(G^\vee) N_\psi$.
Note that 
\[
\Ad(G^\vee)d_{H}(\oo_H^\st) \cap H^\vee = d_{H}(\oo_H^\st), 
\quad
\Ad(G^\vee) N_\psi \cap H^\vee = \Ad(H^\vee) N_\psi. 
\]
(Here, if $H=\SO_{2n}$, we consider $\O_{2n}(\C)$-orbits actually. See Remark \ref{OvsSO}.)
Hence $d_{H}(\oo_H^\st) \geq \Ad(H^\vee) N_\psi$, 
and we conclude that 
$\oo_H^\st \leq d_{H^\vee}(d_{H}(\oo_H^\st)) \leq d_{H^\vee}(\Ad(H^\vee) N_\psi)$.
\end{proof}

\section{Standard endoscopy}
In this section, we consider the standard endoscopy, 
and prove Theorem \ref{main}.
To do this, we will relate the Waldspurger map with the Spaltenstein duality.

\subsection{Waldspurger map}\label{sec.w_map}
Let $H_1 \times H_2$ be an unramified endoscopic group 
of a split classical group $H$. 
We denote by $\xi \colon H_1^\vee \times H_2^\vee \hookrightarrow H^\vee$ the inclusion map. 
\par

For $i \in \{1,2\}$, 
let $\oo_{H_i}^\st \in \NN(\h_i(\overline{F}))$ be a special nilpotent orbit. 
Waldspurger explicitly defined in \cite[Sections XI.6, XI.7]{W_pave}
an orbit
\[
\oo_H^\st = W(\oo_{H_1}^\st, \oo_{H_2}^\st) \in \NN(\h(\overline{F}))
\]
and proved a certain transfer result in \cite[XII.9 Th\'eor\`eme]{W_pave}.
Although his result is quite explicit, 
we only need the following property. 
For any constants $\{c_{\oo_{H_i}} \;|\; \oo_{H_i} \subset \oo_{H_i}^\st\}$ 
such that  
$\sum_{\oo_{H_i} \subset \oo_{H_i}^\st} c_{\oo_{H_i}} \hat\mu_{\oo_{H_i}}$
is stable for $i \in \{1,2\}$, 
there exists a constant $\gamma_{\oo_{H}} = \gamma_{\oo_H}(c_{\oo_{H_1}}, c_{\oo_{H_2}})$ 
for each rational orbit $\oo_{H} \subset \oo_{H}^\st$ such that 
\[
\left( \sum_{\oo_{H} \subset \oo_{H}^\st} \gamma_{\oo_{H}} \hat\mu_{\oo_{H}} \right)
\quad
\text{is the transfer of} 
\quad
\left( \sum_{\oo_{H_1} \subset \oo_{H_1}^\st} c_{\oo_{H_1}} \hat\mu_{\oo_{H_1}} \right)
\otimes
\left( \sum_{\oo_{H_2} \subset \oo_{H_2}^\st} c_{\oo_{H_2}} \hat\mu_{\oo_{H_2}} \right).
\]
Note that in \cite[X. 15 Th\'eor\`eme]{W_pave}, 
Waldspurger determined explicitly 
when $\sum_{\oo_{H_i} \subset \oo_{H_i}^\st} c_{\oo_{H_i}} \hat\mu_{\oo_{H_i}}$ is stable. 
\par

We need the following proposition, 
which will be proven in Section \ref{sec.proof_WS}.
\begin{prop}\label{WS}
For $i \in \{1,2\}$, let $\oo_{H_i}^\st \in \NN(\h_i(\overline{F}))$ be a special nilpotent orbit.
Then 
\[
W(\oo_{H_1}^\st, \oo_{H_2}^\st) 
\leq d_{H^\vee}(\Ad(H^\vee)\xi(d_{H_1}(\oo_{H_1}^\st), d_{H_2}(\oo_{H_2}^\st))). 
\]
\end{prop}

The dimension of $W(\oo_{H_1}^\st, \oo_{H_2}^\st)$ is given in \cite[XI.16 Lemma]{W_pave}.
\begin{lem}\label{dimW}
For $i \in \{1,2\}$, let $\oo_{H_i}^\st \in \NN(\h_i(\overline{F}))$ be a special nilpotent orbit.
Then 
\[
\dim(W(\oo_{H_1}^\st, \oo_{H_2}^\st))
= \dim(\oo_{H_1}^\st) + \dim(\oo_{H_2}^\st) + \dim(\h)-\dim(\h_1)-\dim(\h_2).
\]
\end{lem}

\subsection{Proof of Theorem \ref{main}}\label{sec.proof}
Now we can prove Theorem \ref{main}. 

\begin{proof}[Proof of Theorem \ref{main}]
Part (1) follows from Corollary \ref{cor1}. 
\par

We will prove Theorem \ref{main} (2) by induction on $m = \dim(\psi)$. 
If $\Pi_\psi = \{\pi\}$ is a singleton so that $S\Theta_\psi = \Theta_\pi$, 
then the assertion follows from Corollary \ref{cor1}.
\par

Now we assume that we have
\begin{itemize}
\item
$\oo_H^\st \in \NN(\h(\overline{F}))$ 
with $\dim(\oo_H^\st) \geq \dim(d_{H^\vee}(\Ad(H^\vee)N_\psi))$; 
\item
$\oo_H \in \NN(\h(F))$ with $\oo_H \subset \oo_H^\st$; 
\item
$\pi_0 \in \Pi_\psi$
\end{itemize}
such that $c_{\oo_H}(\pi_0) \not= 0$. 
The goal is to show that $\oo_H^\st = d_{H^\vee}(\Ad(H^\vee)N_\psi)$.
\par

By replacing $\oo_H^\st$ if necessary, 
we may assume that $\oo_H^\st$ is maximal in the set 
\[
\bigcup_{\pi \in \Pi_\psi} \bar\NN(\pi)
\]
with respect to the closure ordering. 
Then since $c_{\oo_H}(\pi)$ is non-negative for each $\pi \in \Pi_\psi$ by M{\oe}glin--Waldspurger \cite{MW}, 
we see that 
\[
c_{\oo_H}(\Pi_\psi) = \sum_{\pi \in \Pi_\psi} c_{\oo_H}(\pi) \not= 0.
\]
Consider 
\[
\Theta_{\Pi_\psi} = \sum_{\pi \in \Pi_\psi} \Theta_\pi  
= \sum_{\oo_H' \in \NN(\h(F))} c_{\oo_H'}(\Pi_\psi) \hat\mu_{\oo_H'}.
\]
Note that $\Theta_{\Pi_\psi} = S\Theta_{\psi, s_\psi}$.
\par

If $s_\psi = 1$, then $S\Theta_\psi = \Theta_{\Pi_\psi}$. 
In this case, Corollary \ref{cor1} implies that $\oo_H^\st = d_{H^\vee}(\Ad(H^\vee)N_\psi)$, 
as desired.
From now, suppose that $s_\psi \not=1$.
Using the eigenspace decomposition for $s_\psi$, 
one can obtain an endoscopic group $H_1 \times H_2$ of $H$
and an $A$-parameter $\psi_i$ for $H_i(F)$ for $i \in \{1,2\}$. 
Since $\det(\psi_1) = \det(\psi_2) = \1$, we see that $H_1 \times H_2$ is split over $F$. 
By \cite[Theorem 2.2.1]{Ar}, we know that $\Theta_{\Pi_\psi} = S\Theta_{\psi, s_\psi}$
is the transfer of $S\Theta_{\psi_1} \otimes S\Theta_{\psi_2}$, 
which is expanded as
\[
\sum_{(\oo_{H_1}^\st, \oo_{H_2}^\st) \in \NN(\h_1(\overline{F})) \times \NN(\h_2(\overline{F}))} 
\left( \sum_{\oo_{H_1} \subset \oo_{H_1}^\st} c_{\oo_{H_1}}(\psi_1) \hat\mu_{\oo_{H_1}}\right)
\otimes 
\left( \sum_{\oo_{H_2} \subset \oo_{H_2}^\st} c_{\oo_{H_2}}(\psi_1) \hat\mu_{\oo_{H_2}}\right).
\] 
Note that in the sum for $(\oo_{H_1}^\st, \oo_{H_2}^\st)$, 
only special nilpotent orbits can contribute 
(see \cite[IX.15 Th\'eor\`eme]{W_pave}). 
Setting $\gamma_{\oo'_{H}} = \gamma_{\oo'_H}(c_{\oo_{H_1}}(\psi_1), c_{\oo_{H_2}}(\psi_2))$, 
the local character expansion of $\Theta_{\Pi_\psi}$ is given by
\[
\Theta_{\Pi_\psi}
= 
\sum_{(\oo_{H_1}^\st, \oo_{H_2}^\st) \in \NN(\h_1(\overline{F})) \times \NN(\h_2(\overline{F}))} 
\sum_{\oo'_H \subset W(\oo_{H_1}^\st, \oo_{H_2}^\st)}
\gamma_{\oo'_H} \hat\mu_{\oo'_H},
\]
where in the last sum, $\oo_H'$ runs over rational orbits contained in $W(\oo_{H_1}^\st, \oo_{H_2}^\st)$.
In particular, $\gamma_{\oo_H} = c_{\oo_H}(\Pi_\psi) \not= 0$. 
Hence we can write $\oo_H^\st = W(\oo_{H_1}^\st, \oo_{H_2}^\st)$ 
for some special nilpotent orbits $\oo_{H_1}^\st$ and $\oo_{H_2}^\st$ 
such that $c_{\oo_{H_1}}(\psi_1) c_{\oo_{H_2}}(\psi_2) \not= 0$
for some $\oo_{H_i} \subset \oo_{H_i}^\st$. 
\par

By the induction hypothesis, we have 
$\dim(\oo_{H_i}^\st) \leq \dim (d_{H_i^\vee}(\Ad(H_i^\vee)N_{\psi_i}))$ for $i \in \{1,2\}$.
Hence by Proposition \ref{WS} and Lemma \ref{dimW}, 
we have 
\begin{align*}
\dim(\oo_H^\st) &= \dim\left(W(\oo_{H_1}^\st, \oo_{H_2}^\st)\right)
\\&\leq \dim\left(
W(d_{H_1^\vee}(\Ad(H_1^\vee)N_{\psi_1}), d_{H_2^\vee}(\Ad(H_2^\vee)N_{\psi_2}))
\right)
\\&\leq 
\dim\left(
d_{H^\vee}(\Ad(H^\vee)\xi(
d_{H_1}(d_{H_1^\vee}(\Ad(H_1^\vee)N_{\psi_1})), d_{H_2}(d_{H_2^\vee}(\Ad(H_2^\vee)N_{\psi_2}))
))
\right)
\\&\leq
\dim\left(
d_{H^\vee}(\Ad(H^\vee)\xi(\Ad(H_1^\vee)N_{\psi_1}, \Ad(H_2^\vee)N_{\psi_2}))
\right)
\\&= 
\dim\left(d_{H^\vee}(\Ad(H^\vee)N_{\psi})\right).
\end{align*}
Here, we used the facts 
that $d_{H_i}(d_{H_i^\vee}(\Ad(H_i^\vee)N_{\psi_i})) \geq \Ad(H_i^\vee)N_{\psi_i}$
and that $d_{H^\vee}$ is an order-reversing map.
Since we assume that $\dim(\oo_H^\st) \geq \dim\left(d_{H^\vee}(\Ad(H^\vee)N_{\psi})\right)$, 
this must be an equality.
Moreover, it together with Lemma \ref{dimW} implies that 
$\oo_{H_i}^\st = d_{H_i^\vee}(\Ad(H_i^\vee)N_{\psi_i})$ for $i \in \{1,2\}$. 
Then by Proposition \ref{WS}, we have
\begin{align*}
\oo_{H^\st} 
&\leq d_{H^\vee}(\Ad(H^\vee)\xi(d_{H_1}(\oo_{H_1}^\st), d_{H_2}(\oo_{H_2}^\st)))
\\&\leq d_{H^\vee}(\Ad(H^\vee)\xi(\Ad(H_1^\vee)N_{\psi_1}, \Ad(H_2^\vee)N_{\psi_2}))
\\&= d_{H^\vee}(\Ad(H^\vee)N_{\psi}).
\end{align*}
Since $\dim(\oo_H^\st) = \dim\left(d_{H^\vee}(\Ad(H^\vee)N_{\psi})\right)$, 
we conclude that $\oo_H^\st = d_{H^\vee}(\Ad(H^\vee)N_\psi)$, 
as desired.
\par

The proof of Theorem \ref{main} (3) is similar. 
Suppose that $\oo_H^\st \in \NN(\h(\overline{F}))$ contains a rational orbit $\oo_H$ 
such that $c_{\oo_H}(\pi_0) \not= 0$ for some $\pi_0 \in \Pi_\psi$. 
We may assume that $\oo_H^\st$ is maximal in $\cup_{\pi \in \Pi_\psi} \bar\NN(\pi)$
so that $c_{\oo_H}(\Pi_\psi) \not= 0$.
If $s_\psi = 1$, then Proposition \ref{refine} implies that $\oo_H^\st \leq d_{H^\vee}(\Ad(H^\vee)N_\psi)$. 
Otherwise, we can write $\oo_H^\st = W(\oo_{H_1}^\st, \oo_{H_2}^\st)$ as above
such that $\oo_{H_i}^\st \leq d_{H_i^\vee}(\Ad(H_i^\vee)N_{\psi_i})$ for $i \in \{1,2\}$ by the induction hypothesis.
By Proposition \ref{WS} and Lemma \ref{Worder} below, we have 
\begin{align*}
\oo_H^\st &= W(\oo_{H_1}^\st, \oo_{H_2}^\st)
\\&\leq W(d_{H_1^\vee}(\Ad(H_1^\vee)N_{\psi_1}), d_{H_2^\vee}(\Ad(H_2^\vee)N_{\psi_2}))
\\&\leq 
d_{H^\vee}(\Ad(H^\vee)\xi(
d_{H_1}(d_{H_1^\vee}(\Ad(H_1^\vee)N_{\psi_1})), d_{H_2}(d_{H_2^\vee}(\Ad(H_2^\vee)N_{\psi_2}))
))
\\&\leq
d_{H^\vee}(\Ad(H^\vee)\xi(\Ad(H_1^\vee)N_{\psi_1}, \Ad(H_2^\vee)N_{\psi_2}))
\\&= 
d_{H^\vee}(\Ad(H^\vee)N_{\psi}).
\end{align*}
This shows the inequality in Conjecture \ref{jiang}.
\par

On the other hand, 
we already know that $d_{H^\vee}(\Ad(H^\vee)N_{\psi}) \in \bar\NN(\pi)$ for some $\pi \in \Pi_\psi$. 
Since $\oo_H^\st \in \bar\NN(\pi)$ implies that 
$\oo_H^\st \leq d_{H^\vee}(\Ad(H^\vee)N_{\psi}) \in \bar\NN(\pi)$, 
if $\oo_H^\st \in \bar\NN(\pi)^{\max}$, then we must have 
$\oo_H^\st = d_{H^\vee}(\Ad(H^\vee)N_{\psi})$. 
This means that $\bar\NN(\pi)^{\max} = \{d_{H^\vee}(\Ad(H^\vee)N_{\psi})\}$.
This completes the proof of Theorem \ref{main} (3).
\end{proof}

\subsection{Preliminaries for proving Proposition \ref{WS}}\label{sec.partitions}
To show Proposition \ref{WS}, 
we recall some notions for partitions. 
\par

A \emph{partition} of a non-negative integer $d$ is 
a sequence $\lambda = (\lambda_1, \lambda_2, \ldots)$ such that 
$\lambda_i$ is a non-negative integers, $\lambda_1 \geq \lambda_2 \geq \cdots$ 
and $\lambda_1+\lambda_2+\cdots = d$.
We also write $\lambda = (\lambda_1, \dots, \lambda_l)$ if $\lambda_{j} = 0$ for $j > l$.
For an integer $k > 0$, we put $c_k(\lambda) = \#\{j \;|\; \lambda_j = k\}$.
Let $\PP(d)$ be the set of partitions of $d$. 
For $\lambda \in \PP(d)$, we denote by $\lambda^t$ the \emph{transpose} of $\lambda$. 
It is defined by 
\[
c_k(\lambda^t) = \lambda_k-\lambda_{k+1}
\]
for $k > 0$. 
\par

For $\lambda_i = (\lambda_{i,1}, \lambda_{i,2}, \ldots) \in \PP(d_i)$ for $i \in \{1,2\}$, 
we define partitions $\lambda_1+\lambda_2$ and $\lambda_1 \cup \lambda_2$ of $d_1+d_2$ by 
\begin{itemize}
\item
$\lambda_1+\lambda_2 = 
(\lambda_{1,1}+\lambda_{2,1}, \lambda_{1,2}+\lambda_{2,2}, \lambda_{1,3}+\lambda_{2,3}, \ldots)$;
\item
$c_k(\lambda_1 \cup \lambda_2) = c_k(\lambda_1)+c_k(\lambda_2)$ 
for each $k > 0$. 
\end{itemize}
For two partitions $\lambda = (\lambda_1, \lambda_2, \ldots)$ and $\mu = (\mu_1, \mu_2, \ldots)$ of $d$, 
we write $\lambda \geq \mu$ if 
\[
\sum_{j \leq j_0} \lambda_j \geq \sum_{j \leq j_0} \mu_j, 
\quad \forall j_0 \geq 1.
\]
By \cite[Lemma 4.3]{CK}, we have the following. 
\begin{enumerate}
\item
$\lambda \geq \mu$ if and only if $\lambda^t \leq \mu^t$.
\item
If $\lambda_i \geq \mu_i$ for $i \in \{1,2\}$, then $\lambda_1 \cup \lambda_2 \geq \mu_1 \cup \mu_2$.
\item
$(\lambda_1 \cup \lambda_2)^t = \lambda_1^t + \lambda_2^t$. 
\item
$(\lambda_1 \cup \lambda_2) + (\mu_1 \cup \mu_2) \geq (\lambda_1+\mu_1) \cup (\lambda_2+\mu_2)$.
\end{enumerate}
\par

We say that 
\begin{itemize}
\item
$\lambda$ is \emph{orthogonal} if $c_k(\lambda)$ is even 
for each even integer $k > 0$; 
\item
$\lambda$ is \emph{symplectic} if $c_k(\lambda)$ is even for each odd integer $k > 0$.
\end{itemize}
We denote by $\PP_\orth(d)$ (\resp $\PP_\symp(d)$) the set of orthogonal (\resp symplectic) partitions of $d$. 
If $d$ is even, we say that $\lambda\in \PP_\orth(d)$ is \emph{very even}, 
if all parts of $\lambda$ are even and each distinct part appears with even multiplicity. 
Decorate every even partition with the subscripts $I$ and $II$, e.g. $(2,2,4,4)_I$ and $(2,2,4,4)_{II}$. 
If $\lambda$ is not very even, a decorated partition is the same as the partition $\lambda$. 
For $d$ even, let $\PP_\orth(d)'$ denote the set of decorated partitions.
\par

A partition $\lambda \in \PP_\orth(2n) \cup \PP_\symp(2n)$ (\resp $\lambda \in \PP_\orth(2n+1)$) is called \emph{special} 
if $\lambda^t \in \PP_\symp(2n)$ (\resp $\lambda \in \PP_\orth(2n+1)$). 
\par

For a split classical group $H$, there is a canonical bijection 
\[
\NN(\h(\overline{F})) \longleftrightarrow 
\left\{
\begin{aligned}
&\PP_\orth(2n+1) \iif H = \SO_{2n+1}, \\
&\PP_\symp(2n) \iif H = \Sp_{2n}, \\
&\PP_\orth(2n)' \iif H = \SO_{2n}.
\end{aligned}
\right. 
\]
This bijection preserves the notions of special and order. Via this bijection, we identify $\NN(\h(\overline{F}))$ 
with the set of partitions satisfying the above conditions. 
\par

We remark that the very even decorated partitions $\lambda_I$ and $\lambda_{II}$ 
are not conjugate under $\SO_{2n}(\overline{F})$, but are conjugate by $\O_{2n}(\overline{F})$. 
This subtlety will not play a role in the combinatorial proofs below, and therefore we may ignore the decorations. 
See also Remark \ref{OvsSO}. 
\par

We write 
\[
d \colon \PP_\orth(2n+1) \leftrightarrows \PP_\symp(2n), 
\quad 
d \colon \PP_\orth(2n) \rightarrow \PP_\orth(2n)
\]
for the Spaltenstein dual maps. 
It is given by 
\[
d(\lambda) = \left\{\begin{aligned}
&((\lambda^t)^-)_C \iif \lambda \in \PP_\orth(2n+1), \\
&((\lambda^t)^+)_B \iif \lambda \in \PP_\symp(2n), \\
&(\lambda^t)_D \iif \lambda \in \PP_\orth(2n), \\
\end{aligned}
\right.
\]
where 
$(\lambda^t)^-$ (\resp $(\lambda^t)^+$) is given from $\lambda^t$ 
by replacing the smallest (\resp largest) positive component with itself minus $1$ (\resp plus $1$),  
and $\mu_X$ denotes the ``$X$-collapse'' of $\mu$ for $X \in \{B,C,D\}$.
Namely, $\mu_B$ (\resp $\mu_C$, $\mu_D$)
is unique largest partition in $\PP_{\orth}(2n+1)$ (\resp $\PP_{\symp}(2n)$, $\PP_\orth(2n)$) 
such that $\mu_X \leq \mu$. 
See \cite[Sections 4.1, 4.3]{CMBO1} for more details. 
\par

Now we recall the Waldspurger map (\cite[Sections XI.6, XI.7]{W_pave}).
Set $d = 2n+1$ if $H = \SO_{2n+1}$, and $d = 2n$ if $H = \Sp_{2n}$ or $H = \SO_{2n}$. 
Let $\lambda_i = (\lambda_{i,1}, \lambda_{i,2}, \ldots) \in \PP_*(d_i)$ be 
the special partition corresponding to a given special nilpotent orbit $\oo_{H_i}^\st \in \NN(\h_i(\overline{F}))$, 
where $* \in \{\orth, \symp\}$. 
Set $\epsilon_i = 1$ (\resp $\epsilon_i = 0$) if $* = \orth$ (\resp $* = \symp$).
\par

We define $J^+$ (\resp $J^-$) by the set of $j \geq 1$ such that 
\begin{itemize}
\item
$j \equiv d+1 \bmod 2$ (\resp $j \equiv d \bmod 2$); 
\item
$\lambda_{1,j} \equiv \epsilon_1 \bmod 2$ and $\lambda_{2,j} \equiv \epsilon_2 \bmod 2$;
\item
$j=1$ or $\lambda_{1,j-1}+\lambda_{2,j-1} > \lambda_{1,j} + \lambda_{2,j}$
(\resp $\lambda_{1,j}+\lambda_{2,j} > \lambda_{1,j+1}+\lambda_{2,j+1}$). 
\end{itemize}
Put $\xi = (\xi_1, \xi_2, \ldots)$ by
\[
\xi_j = \left\{
\begin{aligned}
&0 \iif j \not\in J^+ \cup J^-, \\
&1 \iif j \in J^+, \\
&{-1} \iif j \in J^-.
\end{aligned}
\right.
\]
Then $W(\oo_{H_1}^\st, \oo_{H_2}^\st) \in \NN(\h(\overline{F}))$ 
is defined by the orbit corresponding to the partition $\lambda = \lambda_1+\lambda_2+\xi$, 
i.e., if we write $\lambda = (\lambda_1, \lambda_2, \ldots)$, 
then $\lambda_j = \lambda_{1,j}+\lambda_{2,j}+\xi_j$ for $j \geq 1$. 
In this case, we write $\lambda = W(\lambda_1, \lambda_2)$. 

\subsection{Proof of Proposition \ref{WS}}\label{sec.proof_WS}
Now we show Proposition \ref{WS}. 
Let $\lambda_i \in \PP_*(d_i)$ be the special partition corresponding to $\oo_{H_i}^\st \in \NN(\h_i(\overline{F}))$.
Set $\lambda = W(\lambda_1,\lambda_2) \in \PP_*(d)$, which corresponds to $W(\oo_{H_1}^\st, \oo_{H_2}^\st)$.
Here, $d$ and $d_i$ are the sizes of $H$ and $H_i$, respectively, and $* \in \{\orth, \symp\}$.
Note that $d = d_1+d_2-1$ if $d$ is odd, whereas $d = d_1+d_2$ if $d$ is even.
We set $\epsilon_i = 1$ (\resp $\epsilon_i = 0$) if $\lambda_i \in \PP_\orth(d_i)$ (\resp $\lambda_i \in \PP_\symp(d_i)$).
\par

The orbit 
$d_{H^\vee}(\Ad(H^\vee)\xi(d_{H_1}(\oo_{H_1}^\st), d_{H_2}(\oo_{H_2}^\st)))$ 
in Proposition \ref{WS} corresponds to 
\[
d(d(\lambda_1) \cup d(\lambda_2)).
\]
Since $\lambda \leq d(d(\lambda))$ in general (see \cite[Corollary A3]{BV}), to prove Proposition \ref{WS}, 
it is enough to show that
\[
d(\lambda) \geq d(\lambda_1) \cup d(\lambda_2). 
\]
We will verify this inequality by induction on $d$. 
If $d = 0$ so that $d_1=d_2=0$, the claim is trivial.

\begin{lem}\label{good}
We may assume that 
all nonzero components of $\lambda_i^t$ 
are congruent to $d$ modulo $2$ for $i \in \{1,2\}$. 
\end{lem}
\begin{proof}
By symmetry, it is enough to consider the case where $i=1$. 
Suppose that there is an integer $k > 0$ with $k \not\equiv d \bmod 2$ such that 
$c_{k}(\lambda_1^t) = \lambda_{1,k}-\lambda_{1,k+1} > 0$. 
Since $\lambda_1$ is special, it is even. 
Set 
\[
\lambda_1' = (\lambda_{1,1}-2, \dots, \lambda_{1,k}-2, \lambda_{1,k+1}, \lambda_{1,k+2}, \ldots).
\]
We write $\lambda' = W(\lambda_1',\lambda_2) = \lambda_1'+\lambda_2+\xi'$. 
Then $\xi' = \xi$.
Since $\lambda^t = (\lambda')^t \cup (k,k)$, 
by a case-by-case argument, we see that $d(\lambda) = d(\lambda') \cup (k,k)$.
Similarly, we have $d(\lambda_1) = d(\lambda_1') \cup (k,k)$. 
Hence the inequality $d(\lambda') \geq d(\lambda_1') \cup d(\lambda_2)$
implies $d(\lambda) \geq d(\lambda_1) \cup d(\lambda_2)$.
\par

Therefore, we may replace $\lambda_1$ with $\lambda_1'$. 
Repeating this argument, we may assume that all positive components of $\lambda_1^t$ are equivalent to $d$ modulo $2$. 
\end{proof}

In the rest of this subsection, 
we assume that 
all nonzero components of $\lambda_i^t$ are congruent to $d$ modulo $2$ for $i \in \{1,2\}$. 
\par

Write $\lambda_i = (\lambda_{i,1}, \dots, \lambda_{i,k})$ such that $\lambda_{1,k} + \lambda_{2,k} > 0$. 
Define 
\[
l_i = \left\{
\begin{aligned}
&\max\{j \geq 1 \;|\; \lambda_{i,k-j+1} = \lambda_{i,k}\} \iif H = \SO_d, \\
&\max\{j \geq 1 \;|\; \lambda_{i,j} = \lambda_{i,1}\} \iif H = \Sp_d.
\end{aligned}
\right. 
\]
By Lemma \ref{good}, 
we see that $k \equiv d \bmod 2$, and that $l_i$ is even unless $l_i=k$.
We set $l = \min\{l_1,l_2\}$.
\par

First we consider the case where $l$ is odd. 
Then $H = \SO_d$ and $H_i = \SO_{d_i}$ are odd special orthogonal groups
with $d = d_1+d_2-1$, and $l_1=l_2=k$. 
In this case, we can write $\lambda_i = (\underbrace{a_i,\dots,a_i}_{l})$ 
with $a_i l = d_i$.
Hence 
\[
d(\lambda_i) = (\underbrace{l, \dots, l}_{a_i-1}, l-1)_C = (\underbrace{l, \dots, l}_{a_i-1}, l-1).
\]
On the other hand, since 
\[
\lambda = (\underbrace{a_1+a_2,\dots,a_1+a_2}_{l-1}, a_1+a_2-1), 
\]
we have
\[
d(\lambda) = (\underbrace{l, \dots, l}_{a_1+a_2-1}, l-2)_C = (\underbrace{l, \dots, l}_{a_1+a_2-2}, l-1,l-1)
\]
unless $l=1$ in which case 
$d(\lambda_i) = (\underbrace{1,\dots,1}_{d_i-1})$ and $d(\lambda) = (\underbrace{1,\dots,1}_{d-1})$.
Therefore, in any case, we have
\[
d(\lambda) = d(\lambda_1) \cup d(\lambda_2), 
\]
as desired. 
\par

In the rest, we assume that $l$ is even. 
Set
\[
a_i = \left\{
\begin{aligned}
&\lambda_{i,k} \iif H = \SO_d, \\
&\lambda_{i,1} \iif H = \Sp_d,
\end{aligned}
\right.
\quad
\lambda_i' = \left\{
\begin{aligned}
&(\lambda_{i,1}, \dots, \lambda_{i,k-l}) \iif H = \SO_d, \\
&(\lambda_{i,l+1}, \dots, \lambda_{i,k}) \iif H = \Sp_d.
\end{aligned}
\right.
\]
We write $\lambda' = W(\lambda_1',\lambda'_2) = \lambda_1'+\lambda'_2+\xi'$. 
We have the following.

\begin{itemize}
\item
If $H=\SO_d$, then $\xi_j' = \xi_j$ unless $j \in \{k-l+1, k\}$. 
Moreover, $\xi'_{k-l+1} = \xi'_k = 0$ and $\xi_{k-l+1} = -\xi_k = \delta$, 
where we set
\[
\delta = \left\{
\begin{aligned}
&1 \iif a_1 \equiv a_2 \equiv 1 \bmod 2, \\
&0 \other.
\end{aligned}
\right.
\]

\item
If $H = \Sp_d$, then $\xi_j' = \xi_{j+l}$ for any $j \geq 1$. 
Moreover, $\xi_{2} = \dots = \xi_{l-1} = 0$ and $\xi_1 = -\xi_l = \delta$, 
where we set
\[
\delta = \left\{
\begin{aligned}
&1 \iif a_1 \equiv \epsilon_1 \bmod 2,\; a_2 \equiv \epsilon_2 \bmod 2, \\
&0 \other.
\end{aligned}
\right.
\]
\end{itemize}
\par

Suppose that $\delta = 0$. 
Then $\lambda_i = \lambda'_i \cup (\underbrace{a_i, \dots, a_i}_l)$ 
and $\lambda = \lambda' \cup (\underbrace{a_1+a_2, \dots, a_1+a_2}_l)$.
By a case-by-case consideration, we see that 
\[
d(\lambda_i) = \left(d(\lambda_i')+(\underbrace{l,\dots,l}_{a_i}) \right)_{X_i}, 
\quad
d(\lambda) = \left(d(\lambda')+(\underbrace{l,\dots,l}_{a_1+a_2}) \right)_X
\]
with $X, X_i \in \{B,C,D\}$.
By \cite[Lemma 4.3(4)]{CK}, we have
\begin{align*}
&d(\lambda') \geq d(\lambda_1') \cup d(\lambda_2') 
\\&\implies 
d(\lambda') +(\underbrace{l,\dots,l}_{a_1+a_2}) 
\geq \left(d(\lambda_1') + (\underbrace{l,\dots,l}_{a_1})\right) \cup \left(d(\lambda_2') + (\underbrace{l,\dots,l}_{a_2})\right)
\\&\implies 
d(\lambda') +(\underbrace{l,\dots,l}_{a_1+a_2}) 
\geq d(\lambda_1) \cup d(\lambda_2) 
\\&\implies 
d(\lambda) = 
\left(d(\lambda') +(\underbrace{l,\dots,l}_{a_1+a_2})\right)_X
\geq d(\lambda_1) \cup d(\lambda_2).
\end{align*}
Here, we used the largest property of $X$-collapse.
\par

Suppose that $\delta = 1$. 
Then $\lambda = \lambda' \cup (a_1+a_2+1, \underbrace{a_1+a_2, \dots, a_1+a_2}_{l-2}, a_1+a_2-1)$ 
so that 
\[
\lambda^t = (\lambda')^t + (\underbrace{l,\dots,l}_{a_1+a_2-1}, l-1, 1).
\]
By a case-by-case consideration, we see that 
\[
d(\lambda) = \left(d(\lambda') + (\underbrace{l,\dots,l}_{a_1+a_2-1}, l-1, 1) \right)_X
\]
for some $X \in \{B,C,D\}$.
On the other hand, by the definition of $l$, 
we can find $i \in \{1,2\}$ such that 
$\lambda_{i,k-l} > a_i$ (\resp $a_i > \lambda_{i,l+1}$) if $H = \SO_d$ (\resp $H = \Sp_{d}$). 
Here, when $l = k$ and $H = \SO_d$ (\resp $H = \Sp_d$), 
we formally understand that $\lambda_{i,0} > a_i$ (\resp $a_i > \lambda_{i,k+1}$).
Since $a_i \equiv \epsilon_i \bmod 2$, we see that $\lambda_i = \lambda'_i \cup (\underbrace{a_i,\dots,a_i}_l)$
implies that 
\[
d(\lambda_i) = \left(d(\lambda'_i) + (\underbrace{l,\dots,l}_{a_i-1}, l-1, 1) \right)_{X_i}
\]
for some $X_i \in \{B,C,D\}$.
Then by \cite[Lemma 4.3(4)]{CK}, with $i' \in \{1,2\}$ such that $i' \not= i$, 
we have
\begin{align*}
&d(\lambda') \geq d(\lambda_1') \cup d(\lambda_2') 
\\&\implies 
d(\lambda') +(\underbrace{l,\dots,l}_{a_1+a_2-1},l-1,1) 
\geq \left(d(\lambda_i') + (\underbrace{l,\dots,l}_{a_i-1},l-1,1)\right) 
\cup \left(d(\lambda_{i'}') + (\underbrace{l,\dots,l}_{a_{i'}})\right)
\\&\implies 
d(\lambda') +(\underbrace{l,\dots,l}_{a_1+a_2-1},l-1,1) 
\geq d(\lambda_1) \cup d(\lambda_2) 
\\&\implies 
d(\lambda) = 
\left(d(\lambda') +(\underbrace{l,\dots,l}_{a_1+a_2-1},l-1,1)\right)_X
\geq d(\lambda_1) \cup d(\lambda_2).
\end{align*}
Here, we used the largest property of $X$-collapse.
\par

In both case, we obtained the implication 
\[
d(\lambda') \geq d(\lambda_1') \cup d(\lambda_2') 
\implies 
d(\lambda) \geq d(\lambda_1) \cup d(\lambda_2).
\]
By induction on $d$, we obtain the assertion of Proposition \ref{WS}. 

\subsection{More on Waldspurger's partitions} 
Retain the notation from the previous subsection. 
In particular $(\lambda_1,\lambda_2)$ is a pair of partitions of type $(B,B)$, $(C,D)$, or $(D,D)$, 
and $W(\lambda_1,\lambda_2)$ is the partition of type $B$, $C$, or $D$, respectively, as defined by Waldspurger. 
While $\lambda_1$ and $\lambda_2$ are special, $W(\lambda_1,\lambda_2)$ is not necessarily special. 
Therefore, it is instructive to understand the smallest special partition $\tl\lambda$ 
larger than $W(\lambda_1,\lambda_2)$. 
Notice that Proposition \ref{WS} is also equivalent to the statement 
\[
d(\tl\lambda)\geq d(\lambda_1)\cup d(\lambda_2).
\]
To describe $\tl\lambda$, 
we need to invoke the connection with Lusztig's symbols for Weyl group representations. 
For more details on the below combinatorial statements, 
see for example \cite[Chapter 13]{Car}. 
\par

If $W=W(B_n)=W(C_n)$ is the Weyl group of types $B_n$ and $C_n$, 
an irreducible $W$-representation is parametrized by a bipartition 
\[
\rho = \alpha \times \beta
= (a_0,a_1,\dots,a_k)\times (b_1,b_2,\dots,b_k),
\]
where $0 \leq a_0 \leq a_1 \leq \dots \leq a_k$ and $0 \leq b_1 \leq b_2 \leq \dots \leq b_k$
with $\sum a_i+\sum b_i=n$. 
The \emph{symbol} (or \emph{a-symbol}) of $\rho$ is
\[
A(\rho)=\left(\begin{matrix}a_0&&a_1+1& &\dots &&{a_k+k}\\&{b_1}&&b_2+1&&b_k+(k-1)\end{matrix}\right).
\]
Two symbols are regarded the same if one is obtained from the other by the transformation 
$\left(\begin{matrix}\underline a\\\underline b\end{matrix}\right)
\to \left(\begin{matrix}0&\underline a+1\\0&\underline b+1\end{matrix}\right)$. 
For example, 
$\left(\begin{matrix}0&&1&&2\\&2&&3\end{matrix}\right)
\equiv \left(\begin{matrix}0&&1&&2&&3\\&0&&3&&4\end{matrix}\right)$.
Two symbols $A(\rho)$ and $A(\rho')$ are in the same \emph{family} 
if one can be obtained from the other by permuting the entries so that each row is strictly increasing. 
In each family, there is a unique \emph{special} symbol, 
the one for which
\[
a_0 \leq b_1 \leq a_1+1 \leq b_2+1 \leq \dots \leq b_k+(k-1) \leq a_k+k.
\]
\par

If $W'=W(D_n)$ is the Weyl group of type $D_n$, 
the irreducible $W'$-representations are also parameterized by bipartitions of $n$, 
$\rho=\alpha\times \beta$, as for $W$, 
except $a_0=0$, and 
\begin{itemize}
\item 
if $\alpha=\beta$, 
then there are two different representations, 
labeled $(\alpha \times \alpha)_{I}$ and $(\alpha \times \alpha)_{II}$;
\item 
if $\alpha \neq \beta$, 
then $\alpha \times \beta = \beta \times \alpha$.
\end{itemize}
The difference comes from restricting a representation of $W$ to the index-two subgroup $W'$. 
The \emph{symbol} of type $D$ for $\rho$ is
\[
A^D(\rho)=\left(\begin{matrix}a_1&a_2+1 &\dots &a_k+(k-1)\\b_1&b_2+1&\dots&b_k+(k-1)\end{matrix}\right),
\]
with the caveat that if the two rows are identical then there are two decorated symbols $I,II$, 
and if the rows are distinct, flipping the rows gives the same symbol. 
Moreover, as in the case of types $B/C$, 
two symbols are regarded the same if one is obtained from the other by the transformation 
$\left(\begin{matrix}\underline a\\\underline b\end{matrix}\right)
\to \left(\begin{matrix}0&\underline a+1\\0&\underline b+1\end{matrix}\right)$.
Just as before, in each type $D$ family, 
there is a unique \emph{special} symbol; 
we choose the one for which
\[
b_1\leq a_1 \leq b_2+1 \leq \dots \leq b_k+(k-1) \leq a_k+(k-1).
\]
\par

Suppose $A(\rho)$ (or $A^D(\rho)$) is a special symbol. 
The corresponding special partition is obtained as follows:
\begin{description}
\item[Type $B$]
\[
\lambda(\rho)^B=\bigsqcup_{i\notin I^B}\{2a_{i-1}+1,2b_i-1\}\sqcup\bigsqcup_{i\in I^B}\{2a_{i-1},2b_i\}\sqcup\{2a_k+1\},
\]
where $I^B=\{1\le i\le k\mid b_i=a_{i-1}\}$;
\item[Type $C$]
\[
\lambda(\rho)^C=\{2a_0\}\sqcup\bigsqcup_{i\notin I^C} \{2a_i,2b_i\}\sqcup \bigsqcup_{i\in I^C}\{2a_i+1,2b_i-1\},
\]
where $I^C=\{1\le i\le k\mid b_i=a_i+1\}$;
\item[Type $D$]
\[
\lambda(\rho)^D=\bigsqcup_{i\notin I^D}\{2b_i+1,2a_i-1\}\sqcup\bigsqcup_{i\in I^D}\{2b_i,2a_i\},
\]
where $I^D=\{1\le i\le k\mid b_i=a_{i}\}$.
\end{description}
\par

The key fact for computing the special partition $\tl\lambda$ is given by Waldspurger \cite[XI.15 Lemme]{W_pave}.
\begin{lem}\label{l:wal}
Suppose the (special) Springer representations attached to $\lambda_1$ and $\lambda_2$ 
are $\alpha \times \beta$ and $\alpha' \times \beta'$, respectively. 
Then the Springer representation corresponding to $W(\lambda_1,\lambda_2)$ (and the trivial local system) 
is $(\alpha+\alpha') \times (\beta+\beta')$.
\end{lem}

The failure of $W(\lambda_1,\lambda_2)$ to be special comes from the fact that 
the symbol for $(\alpha+\alpha') \times (\beta+\beta')$ may not be special, 
even though  both symbols for $\alpha\times\beta$ and $\alpha'\times\beta'$ are special.
A case-by-case analysis immediately leads the following proposition.

\begin{prop}
The unique special Weyl group representation 
$\tl\rho = (c_0,c_1,\dots,c_k) \times (d_1,d_2,\dots,d_k)$ 
in the same family as the Springer representation $(\alpha+\alpha')\times(\beta+\beta')$ of $W(\lambda_1,\lambda_2)$ 
is given as follows:
\begin{description}
\item[Type $(B,B)$]
\[
(c_i,d_i) = \left\{
\begin{aligned}
&(a_i+a_i',b_i+b_i') \iif i\notin J(\tl\lambda), \\
&(a_i+a_i'+1,b_i+b_i'-1) \iif i\in J(\tl\lambda),
\end{aligned}
\right.
\]
where $J(\tl\lambda)=\{1\le i\le k\mid b_i=a_i+1\text{ and }b_i'=a_i'+1\}$;
\item[Type $(C,D)$]
\[
(c_i,d_i) = \left\{
\begin{aligned}
&(a_i+a_i',b_i+b_i') \iif i\notin J(\tl\lambda), \\
&(b_i+b_i',a_{i-1}+a'_{i-1}) \iif i\in J(\tl\lambda),
\end{aligned}
\right.
\]
where $J(\tl\lambda)=\{1\le i\le k\mid b_i=a_{i-1}\text{ and }b_i'=a_{i-1}'-1\}$;
\item[Type $(D,D)$]
\[
(c_i,d_i) = \left\{
\begin{aligned}
&(a_i+a_i',b_i+b_i') \iif i\notin J(\tl\lambda),\\
&(b_i+b_i'+1,a_{i-1}+a'_{i-1}-1) \iif i\in J(\tl\lambda),
\end{aligned}
\right.
\]
where $J(\tl\lambda)=\{2\le i\le k\mid b_i=a_{i-1}-1\text{ and }b_i'=a_{i-1}'-1\}$.
\end{description}
\end{prop}

The special partition $\tl\lambda$ can then be obtained from $\tl\rho$ 
via the procedure described before Lemma \ref{l:wal}.
\par

\begin{rem}
The Springer representation corresponding to Waldspurger's orbit $W(\lambda_1,\lambda_2)$ 
is in fact the $j$-induced representation (truncated induction) from the Springer representations 
attached to $\lambda_1$ and $\lambda_2$. 
This can be seen at once from the explicit combinatorial description in \cite[Sections 4.5 (c), 5.3 (b), 6.3 (b)]{Lu}.
\end{rem}

Finally, we need a result on the relation between the closure order and the combinatorics of bipartitions. 
If $\alpha\times\beta=(a_0,a_1,\dots,a_k)\times (b_1,b_2,\dots,b_k)$ 
and $\gamma\times\delta=(p_0,p_1,\dots,p_k)\times (q_1,q_2,\dots,q_k)$ are two bipartitions of $n$, 
we write
$\alpha\times\beta\le \gamma\times\delta$
if
\begin{align*}
a_k+b_k+a_{k-1}+b_{k-1}+\dots+a_i+b_i&\le p_k+q_k+p_{k-1}+q_{k-1}+\dots+p_i+q_i,\text{ and}\\
a_k+b_k+a_{k-1}+b_{k-1}+\dots+a_i&\le p_k+q_k+p_{k-1}+q_{k-1}+\dots+p_i,
\end{align*}
for all $i$. 
\par

The following relation is proved in \cite[Propositions 2.1(3) and 2.12(3)]{AHS} for types $B$ and $C$, 
and the proof for type $D$ is the same.
\begin{prop}\label{p:achar}
Let $\lambda,\lambda' \in \PP_*(d)$ be given partitions with $* \in \{\orth, \symp\}$. 
Let $\alpha\times\beta$ and $\gamma\times\delta$ be the bipartitions corresponding to 
the Springer representations defined by $\lambda$ and $\lambda'$, respectively, and the trivial local systems. 
Then
\[
\lambda\le \lambda'\text{ if and only if }\alpha\times\beta\le \gamma\times\delta. 
\]
\end{prop}

We can now prove the compatibility of Waldspurger's construction with the closure order. 
\begin{lem}\label{Worder}
Let $\lambda_i,\lambda_i' \in \PP_*(d_i)$ be special partitions
for $i \in \{1,2\}$ and $* \in \{\orth, \symp\}$. 
If $\lambda_1 \leq \lambda_1'$ and $\lambda_2 \leq \lambda_2'$, 
then 
\[
W(\lambda_1, \lambda_2) \leq W(\lambda_1', \lambda_2'). 
\]
\end{lem}
\begin{proof}
Suppose the Springer representations attached to $\lambda_1$ and $\lambda_2$ 
are $\alpha\times\beta$ and $\alpha'\times \beta'$, respectively. 
By Lemma \ref{l:wal}, 
the Springer representation corresponding to $W(\lambda_1,\lambda_2)$  is $(\alpha+\alpha')\times(\beta+\beta')$.
Similarly, if the Springer representations attached to $\lambda_1'$ and $\lambda_2'$ 
are $\gamma\times\delta$ and $\gamma'\times \delta'$, respectively, 
then the Springer representation corresponding to $W(\lambda_1',\lambda_2')$ 
is $(\gamma+\gamma')\times(\delta+\delta')$.
\par

It is immediate that if $\alpha\times\beta\le \gamma\times\delta$ and $\alpha'\times\beta'\le \gamma'\times\delta'$, 
then also $(\alpha+\alpha')\times(\beta+\beta')\le (\gamma+\gamma')\times(\delta+\delta')$. 
The claim now follows from Proposition \ref{p:achar}.
\end{proof}

\subsection*{Acknowledgments}
We are deeply grateful to Jean-Loup Waldspurger for sharing his expertise on the endoscopic transfer.
The first author was partially supported by JSPS KAKENHI Grant Number 23K12946.



\begin{thebibliography}{30}
\bibitem{AHS}
{P. N. Achar, A. Henderson and E. Sommers},
{\em Pieces of nilpotent cones for classical groups}. 
\emph{Represent. Theory} {\bf15} (2011), 584--616.

\bibitem{ABV} 
{J. Adams, D. Barbasch and D. A. Vogan,~Jr.}, 
{\em The Langlands classification and irreducible characters for real reductive groups}. 
\emph{Progress in Mathematics,} {\bf104}. Birkh\"auser Boston, Inc., Boston, MA, 1992. xii+318 pp.

\bibitem{Ar}
{J. Arthur}, 
{\em The endoscopic classification of representations. Orthogonal and symplectic groups}. 
\emph{American Mathematical Society Colloquium Publications,} {\bf61}. 
American Mathematical Society, Providence, RI, 2013. xviii+590 pp.

\bibitem{AGIKMS}
{H. Atobe, W. T. Gan, A. Ichino, T. Kaletha, A. M\'inguez and S. W. Shin},
{\em Local Intertwining Relations and Co-tempered $A$-packets of Classical Groups}.
Preprint 2025, arXiv:2410.13504v2.

\bibitem{BV}
{D. Barbasch and D. A. Vogan,~Jr.}, 
{\em Unipotent representations of complex semisimple groups}. 
\emph{Ann. of Math. (2)} {\bf121} (1985), no.~1, 41--110.

\bibitem{Car}
{R. W. Carter}, 
{\em Finite groups of Lie type. Conjugacy classes and complex characters}. 
Reprint of the 1985 original. 
\emph{Wiley Classics Library}. 
A Wiley-Interscience Publication. John Wiley \& Sons, Ltd., Chichester, 1993. xii+544 pp.

\bibitem{CK}
{D. Ciubotaru and J.-L. Kim}, 
{\em The wavefront set: bounds for the Langlands parameter}. 
\emph{Math. Ann.} {\bf393} (2025), no.~2, 1827--1861.

\bibitem{CMBO2}
{D. Ciubotaru, L. Mason-Brown and E. Okada}, 
{\em Wavefront sets of unipotent representations of reductive $p$-adic groups II}. 
\emph{J. Reine Angew. Math.} {\bf823} (2025), 191--253.

\bibitem{CMBO1}
{D. Ciubotaru, L. Mason-Brown and E. Okada}, 
{\em Wavefront sets of unipotent representations of reductive $p$-adic groups I}. 
Preprint 2001, arXiv:2112.14354v5, 
to appear in \emph{Amer. Jour. Math.}

\bibitem{C}
{L. Clozel}, 
{\em Characters of nonconnected, reductive $p$-adic groups}.
\emph{Canad. J. Math.} {\bf39} (1987), no.~1, 149--167.

\bibitem{GLLS}
{F. Gao, B. Liu, C.-H. Lo and F. Shahidi}, 
{\em Covering Barbasch-Vogan duality and wavefront sets of genuine representations}.
Preprint 2025, arXiv:2511.14750v2.

\bibitem{HC}
{Harish-Chandra}, 
{\em Admissible invariant distributions on reductive p-adic groups}. 
With a preface and notes by Stephen DeBacker and Paul J. Sally, Jr. 
\emph{University Lecture Series,} {\bf16}. 
\emph{American Mathematical Society, Providence, RI,} 1999. xiv+97 pp.

\bibitem{HLLS}
{A. Hazeltine, B. Liu, C.-H. Lo and F. Shahidi},
{\em On the upper bound of wavefront sets of representations of $p$-adic groups}.
Preprint 2024, arXiv:2403.11976v2.

\bibitem{HII}
{K. Hiraga, A. Ichino and T. Ikeda}, 
{\em Formal degrees and adjoint $\gamma$-factors.} 
\emph{J. Amer. Math. Soc.} {\bf21} (2008), no.~1, 28--304. 

\bibitem{J}
{D. Jiang}, 
{\em Automorphic integral transforms for classical groups I: Endoscopy correspondences}. 
Automorphic forms and related geometry: assessing the legacy of I. I. Piatetski-Shapiro, 179--242,
\emph{Contemp. Math.,} {\bf614}, 
\emph{Amer. Math. Soc., Providence, RI,} 2014.

\bibitem{K}
{T. Konno}, 
{\em Twisted endoscopy and the generic packet conjecture.} 
\emph{Israel J. Math.} {\bf129} (2002), 253--289.

\bibitem{LS}
{B. Liu and F. Shahidi}, 
{\em The Jiang conjecture on the wavefront sets of local Arthur packets}.
Preprint 2025, arXiv:2503.05343v1.

\bibitem{Lu}
G. Lusztig, 
{\em Unipotent classes and special Weyl group representations}, \emph{J. Algebra} {\bf 321} (2009), 3418--3449.

\bibitem{MW}
{C. M{\oe}glin and J.-L. Waldspurger}, 
{\em Mod\`eles de Whittaker d\'eg\'en\'er\'es pour des groupes $p$-adiques}. 
\emph{Math. Z.} {\bf196} (1987), no.~3, 427--452.

\bibitem{Sp}
{N. Spaltenstein}, 
{\em Classes unipotentes et sous-groupes de Borel}. 
\emph{Lecture Notes in Mathematics,} {\bf946}. 
\emph{Springer-Verlag, Berlin-New York,} 1982. ix+259 pp.

\bibitem{V}
{S. Varma}, 
{\em On descent and the generic packet conjecture.} 
\emph{Forum Math.} {\bf29} (2017), no.~1, 111--155. 

\bibitem{W_pave}
{J.-L. Waldspurger}, 
{\em Int\'egrales orbitales nilpotentes et endoscopie pour les groupes classiques non ramifi\'es}. 
\emph{Ast\'erisque} No.~{\bf269} (2001), vi+449 pp.

\end{thebibliography}
\end{document}